\documentclass[reqno]{amsart}

\usepackage[usenames]{color}
\usepackage{enumerate}
\usepackage{amssymb}

\usepackage{amsmath,amsthm,amsfonts,amssymb,enumerate,amscd}
\usepackage{tikz}
\usepackage{multicol}
\usepackage{verbatim}
\usetikzlibrary{arrows}
\usetikzlibrary{shapes,snakes}

\newtheorem{teorema}{Theorem}[section] 

\newtheorem{lema}[teorema]{Lemma}     

\newtheorem{corollary}[teorema]{Corollary}
\newtheorem{proposition}[teorema]{Proposition}
\newtheorem{proposicion}[teorema]{Proposition}
\newtheorem{definition}[teorema]{Definition}
\newtheorem{definicion}[teorema]{Definition}

\newtheorem{example}[teorema]{Example}

\newtheorem{remark}[teorema]{Remark}

\newcommand{\g}{\gamma}
\newcommand{\ga}{\gamma}

\title[]{Arithmetic progressions and its applications \\
to $(m,q)$-isometries: a survey}

\author[]{Teresa Berm\'udez} 

\author[]{Antonio Martin\'on} 

\author[]{Juan Agust\'in Noda \\ University of La Laguna} \email{tbermude@ull.es, anmarce@ull.es, joannoda@gmail.com}

\address{Departamento de An\'{a}lisis Matem\'{a}tico\\   Universidad de La Laguna\\   38271 La Laguna (Tenerife), Spain}

\date{\today}



\begin{document}

\maketitle


\begin{abstract}

In this paper we collect some results about arithmetic progressions of higher order, also called polynomial sequences. Those results are applied to $(m,q)$-isometric maps.

\end{abstract}


\;

\;


\tableofcontents


\;

\;

\section{INTRODUCTION}

In this paper we collect some results about arithmetic progressions of higher order, also called polynomial sequences. The papers of J. Alonso \cite{alonso} and V. Dlab \cite{dlab}  are dedicated to this topic.

We consider arithmetic progressions on commutative groups. A sequence $(a_n)_{n \geq 0}$ in a group $G$ is an {\it arithmetic progression of order $h$} if
$$
\sum_{k=0}^{h+1} (-1)^{h+1-k} {h \choose k}  a_{n+k} = 0 \; ,
$$
for any integer $n \geq 0$; equivalently, there exists a polynomial $p_a$, with coefficients in $G$ of degree less or equal to $h$, such that $p_a (n) = a_n$, for any $n \geq 0$; that is, there are $\ga_{h}, \ga_{h-1}, ..., \ga_{1}, \ga_{0}$ in $G$ such that
$$
a_n = \ga_{h} n^{h} + \ga_{h-1} n^{h-1} + \cdots + \ga_{2} n^{2} + \ga_{1} n + \ga_0 \; ,
$$
for any $n \geq 0$. An arithmetic progression of order $h$ is of {\it strict order $h$} if $h=0$ or if $h >1$ and it is not of order $h-1$.

We pay attention in certain aspects of the theory of arithmetic progressions which are related with $(m,q)$-isometric maps.

J. Agler \cite{agler} introduced the notion of $m$-isometry
for a positive integer $m$: an operator $T$ acting on a Hilbert space $H$ is an {\it $m$-isometry} if
$$
\sum_{k=0}^m (-1)^{m-k} {m\choose k}  T^{*k} T^{k} = 0 \; ,
$$
where $T^{*}$ denotes the adjoint operator of $T$. In  \cite{as1}, \cite{as2} and \cite{as3} the $m$-isometries are intensively studied.

The above definition of $m$-isometry is equivalent to that $(T^{*n}T^n)_{n \geq 0}$ is an arithmetic progression of order $m-1$ in the algebra $L(H)$ of all (bounded linear) operators on $H$. Several authors have extended the concept of $m$-isometry to the setting of Banach spaces. For more details see \cite{bayart}, \cite{bot}, \cite{hms} and \cite {sid ahmed}. In \cite{bmmu} it was extended to metric spaces.

Now we summarize the contain  of this paper.

\;

{\bf Section 2.}  We consider sequences in the setting of groups. First we study the {\it difference operator} $D$ which acts on a sequence $a = (a_n)_{n \geq 0}$ in $G$ in the following way:
$$
D a = (a_1 - a_0 , a_2 - a_1 , a_3 - a_2 ... ) \; .
$$
The expression of the general term of the sequence $D^h a$, where $D^h$ is the $h$ power of $D$ is given by
$$
D^h a_n = \sum_{k=0}^h (-1)^{h-k} {h \choose k} a_{n+k} \; .
$$

We collect some  combinatorial results which are necessary and obtain an expression of the general term of any sequence $a$: for $n=0,1,2,3...$,
$$
a_n = \sum_{k=0}^n  {n \choose k} D^k a_{0} \; .
$$

\;

{\bf Section 3.} If $a$ is an arithmetic progression of order $h$, then we obtain that
$$
a_n = \sum_{k=0}^h {n \choose k} D^k a_0 \; .
$$
Other expressions for $a_n$ are obtained.

We consider a double sequence $(a_{i,j})_{i,j \geq 0}$ such that its files and columns are arithmetic progressions of order $k$ and $h$, respectively. Then we obtain that the diagonal sequence $(a_{i,i})_{i \geq 0}$ is an arithmetic progression of order $k+h$.

We finalize this section with a perturbation result:  let $x,y $ in a ring $R$ such that $(y^k x^k)_{k \geq 0}$ is an arithmetic progression of order $h$ and let $a,b \in R$ such that $a^n =0$, $b^m =0$, $ax=xa$ and $by=yb$; then the sequence  $((y+b)^k (x+a)^k)_{k \geq 0}$ is an arithmetic progression of order $n+m+h-2$.

\;

{\bf Section 4.}  In this section we work with numerical arithmetic progressions. We recall some results about recursive equations and prove that if $(a_{cn})_{n \geq 0}$ and $(a_{dn})_{n \geq 0}$ are arithmetic progressions, then it results that $(a_{en})_{n \geq 0}$ is also an arithmetic progression, where $e$ is the greatest common divisor of $c$ and $d$.

Moreover we obtain that every arithmetic progression of positive real numbers is eventually increasing; that is, $a_n \leq a_{n+1}$, for $n$ large enough.

The main results of this section are referred to powers of positive sequences. For example, if $a= (a_n)_{n \geq 0}$ is a positive sequence, $q,r >0$ real numbers, $k,h \geq 0$ integer numbers, $a^q$ is an arithmetic progression of order $k$ and $a^r$ is an arithmetic progression of order $h$, then $rk = hq$. From this result we obtain the next classification. Given a positive sequence $a$, then it verifies exactly one of the following assertions:
\begin{enumerate}
\item $a^q$ is not an arithmetic progression of order $h$, for all real $q>0$ and all integer $h \geq 0$.
\item $a$ is constant; hence $a^q$ is an arithmetic progression of order $h$, for all real $q>0$ and all integer $h \geq 0$.
\item There are unique integer $\ell \geq 1$ and real $s > 0$ such that, for every $k = 1,2,3...$, the sequence $a^{ks}$ is an arithmetic progression of strict order $k\ell$; moreover, if $a^q$ is an arithmetic progression of strict order $h$, then $q=ks$ and $h=k\ell$, for some $k=1,2,3...$
\end{enumerate}

\;

{\bf Section 5.} In the last section we apply the results obtained in Sections 3 and 4 to $(m,q)$-isometries.

Let $E$ be  a  metric space, $d$ its distance, $T: E \longrightarrow E$ a map, an integer $m \geq 1$  and a real $q>0$. Then $T$ is an $(m,q)$-isometry if and only if, for all $x,y \in E$, the sequence $(d(T^n x, T^{n}y)^q)_{n \geq 0}$ is an arithmetic progression of order less or equal to $m-1$.

Any power of an $(m,q)$-isometry is also an $(m,q)$-isometry. It is possible to say more: if $T^c$ is an $(m,q)$-isometry and $T^d$ is an $(\ell,q)$-isometry,  then $T^e$ is an $(h,q)$-isometry, where $e$ is the greatest common divisor of $c$ and $d$, and $h$ is the minimum of $m$ and $\ell$.

Using a result about arithmetic progressions allow us to prove that if $S,T: E \longrightarrow E$ are maps on a metric space that commute, $T$ is an $(n,q)$-isometry and $S$ is an $(m,q)$-isometry, then $ST$ is an $(m+n-1,q)$-isometry.

An operator $T : H \longrightarrow H$  on a Hilbert space $H$ is an {\it $m$-isometry} if it is an $(m,2)$-isometry. We apply the perturbation result for arithmetic progressions on a ring to prove that if $T \in L(H)$ is an $m$-isometry and $Q \in L(H)$ is $n$-nilpotent operator such that $TQ=QT$, then $T+Q$ is an $(2n+m-2)$-isometry.

Finally we consider $n$-invertible operators: it is said that the operator $S$ is a {\it left $n$-inverse} of the operator $T$ if
$$
\sum_{k=0}^{n} (-1)^{n-k} {n \choose k}  S^k T^k = 0 \; ;
$$
if that equation holds, it is said that $T$ is a {\it right $n$-inverse} of $S$. This is equivalent to say that the sequence  $(S^{k}T^k)_{k \geq 0}$ is an arithmetic progression of order $n-1$ in $L(H)$.

\;

\;


Mainly, the results which we present are taken or inspired from the paper of P. Hoffmann, M. Mackey and M. O. Searc\'oid \cite{hms}, and from \cite{bdm}, \cite{bmm}, \cite{bmmu}, \cite{bmmn}, \cite{bmmn 2}, \cite{bmne}, \cite{bmno} and \cite{bmno 2}.

We have included proofs in order to make the material as self contained as possible.

\;

\;

\section{SEQUENCES IN GROUPS}

\;

In this section $G$ denotes a commutative group and we denote additively its operation.

\;

\subsection{The difference operator}

\ \par

\;

Let $a = (a_n)_{n \geq 0}$ be a sequence in $G$. The difference sequence $Da$ of $a$ is defined by
$$
D a = (a_1 - a_0 , a_2 - a_1 , a_3 - a_2 ... ) \; . \;
$$
The powers of the difference operator $D$ are defined in the following way: $D^0 a := a$ and, for  $h=1, 2, 3...$,
$$
D^h a := (D^h a_n)_{n \geq 0} \quad \mbox{ where } \quad D^h a_n := D^{h-1} a_{n+1} - D^{h-1} a_{n} \; .
$$
For example, $D^2 a = (a_2 - 2 a_1 + a_0 , a_3 - 2 a_2 + a_1 , a_4 - 2 a_3 + a_2 ... )$.

\begin{teorema} Let $a$ be a sequence in $G$. Then, for $h,n=0,1,2,3...$,
\begin{equation}\label{diff op}
D^h a_n = \sum_{k=0}^h (-1)^{h-k} {h \choose k} a_{n+k} \; .
\end{equation}
\end{teorema}

\begin{proof} It is clear that Equality (\ref{diff op}) is true for $h=0$ and $h=1$. Assume that it is valid for $h-1$ and we will prove that is also valid for $h$. We have that
\begin{eqnarray}
D^h a_n &=& D^{h-1} a_{n+1} - D^{h-1} a_{n} \nonumber \\
&=& \sum_{k=0}^{h-1} (-1)^{h-1-k}  {h-1 \choose k} a_{n+1+k} -  \sum_{k=0}^{h-1} (-1)^{h-1-k} {h-1 \choose k} a_{n+k} \nonumber \\
&=&  \sum_{k=0}^{h-2} (-1)^{h-1-k} {h-1 \choose k} a_{n+1+k} + a_{n+h} + (-1)^{h} a_{n} +  \sum_{k=1}^{h-1} (-1)^{h-k}  {h-1 \choose k} a_{n+k} \nonumber \\
&=& (-1)^{h} a_n + \sum_{k=1}^{h-1} (-1)^{h-k} \left(  {h-1 \choose k-1} + {h-1 \choose k} \right) a_{n+k} + a_{n+h} \nonumber \\
&=& \sum_{k=0}^h (-1)^{h-k} {h \choose k} a_{n+k} \; . \nonumber
\end{eqnarray}
So the proof is finished.
\end{proof}

\;

\subsection{Some combinatorial results}

\ \par

\;

We apply the next combinatorial results in the following. The statements and their proofs are in previous papers, but we have included for sake of completeness.

Notice that for integers $n,k \geq 0$,
$$
{n \choose k} = \frac{n(n-1) \cdots (n-k+1)}{k!} \; ,
$$
being $\displaystyle{{n \choose k} = 0}$ if $n < k$.

\begin{lema}\label{lema} \cite[Lemma 3.1]{bmmu} Let $(x_j)_{j \geq 0}$ be a sequence in a group $G$, and let $(e_k)_{k \geq 0}$ be a sequence of integers and  $(c_{k,j})_{k,j \geq 0}$ be a double sequence of integers. Then
$$
\sum_{k=0}^n e_k \sum_{j=0}^k c_{k,j} x_j = \sum_{j=0}^n x_j \sum_{k=j}^n c_{k,j} e_k  \; .
$$
for any $n=0,1,2...$
\end{lema}

\begin{proof} Note that both expressions are equal to
$$
 \sum_{0 \leq j \leq k \leq n}    e_k c_{k,j} x_j \; ,
$$
so the proof is completed.
\end{proof}

\begin{lema}\label{lema combinatorio 1} \cite[Lemma 3.2]{bmmu}  Let $i,j$ be integers with $0 \leq j < i$. Then
\begin{equation}\label{fla combinatoria 1}
\sum_{h=0}^j (-1)^h {i \choose h} = (-1)^j {i-1 \choose j} \; .
\end{equation}
\end{lema}

\begin{proof} By induction on $i$:

(1) Case $i = j+1$. We have that
$$
\sum_{h=0}^j (-1)^h {j+1 \choose h} = -(-1)^{j+1}  {j+1 \choose j+1} = (-1)^j  \; ,
$$
so (\ref{fla combinatoria 1}) holds.

(2) Assume that (\ref{fla combinatoria 1}) is true for certain $i \geq j+1$ and we prove that it is also true for $i+1$. Indeed,
$$
\sum_{h=0}^j (-1)^h {i+1 \choose h} = {i+1 \choose 0} - \left[ {i \choose 0} + {i \choose 1} \right] +  \cdots + (-1)^{j} \left[ {i \choose j-1} + {i \choose j} \right] = (-1)^{j}   {i \choose j} \; .
$$
This finishes the proof.
\end{proof}

\begin{lema}\label{lema combinatorio 2} \cite[Lemma 3.3]{bmmu} Let $n$, $h$ and $k$ be integers such that $0 \leq k \leq h  < n$.  Then
\begin{equation}\label{fla combinatoria 2}
\sum_{j=k}^{h} (-1)^{j-k} {n \choose j} {j \choose k} = (-1)^{h-k} \frac{n (n-1) \cdots\overbrace{(n-k)}\cdots(n-h)}{k!(h-k)!} \; ,
\end{equation}
where $\overbrace{(n-k)}$ denotes that the factor $(n-k)$ is  omitted.
\end{lema}

\begin{proof} Notice that
\begin{eqnarray}
A &:=& \sum_{j=k}^{h} (-1)^{j-k} {n \choose j} {j \choose k} \nonumber \\
&=& \sum_{j=k}^{h} (-1)^{j-k} \frac{n (n-1) \cdots (n-k+1) (n-k)!}{k! (n-j)! (j-k)!} \nonumber \\
&=& \frac{n (n-1) \cdots (n-k+1)}{k!} \sum_{j=0}^{h-k} (-1)^{j} {n-k \choose j} \; . \nonumber
\end{eqnarray}
Applying (\ref{fla combinatoria 1}) we obtain
$$
A = \frac{n (n-1) \cdots (n-k+1)}{k!} (-1)^{h-k} {n-k-1 \choose h-k} \; .
$$
On the other hand,
\begin{eqnarray}
B &:=& (-1)^{h-k} \frac{n (n-1) \cdots\overbrace{(n-k)}\cdots(n-h)}{k!(h-k)!} \nonumber \\[0.7pc]
&=& (-1)^{h-k} \frac{n (n-1) \cdots (n-k+1) (n-k-1)!}{k!(h-k)! (n-h-1)!} \nonumber \\[0.7pc]
&=& (-1)^{h-k} \frac{n (n-1) \cdots (n-k+1)}{k!} {n-k-1 \choose h-k} \; . \nonumber
\end{eqnarray}
So we have the equality $A=B$, which proves the statement.
\end{proof}

\begin{lema}\label{lema combinatorio 3} \cite[Lemma 3.6]{bmne} Let $n,h$ be non-negative integers.  Then
\begin{equation}\label{fla combinatoria 3}
\sum_{k=0}^{h} (-1)^{h-k}  \frac{n (n-1) \cdots\overbrace{(n-k)}\cdots(n-h)}{k!(h-k)!} = 1 \; .
\end{equation}
\end{lema}

\begin{proof} Denote by $R(n,h)$ the left hand side of (\ref{fla combinatoria 3}). Note that $R(n,h)$ is a polynomial in $n$ of degree at most $h$. Moreover, for $k= 0,1,...,h$, we have that
$$
R(k,h) = (-1)^{h-k}  \frac{k \cdots 1 \overbrace{0} (-1) \cdots  (-h+k)}{k! (h-k)!} =1 \; .
$$
Hence $R(k,h)=1$ for $h+1$ different values, so $R(n,h)=1$.
\end{proof}

\begin{remark} {\rm Note that
\begin{eqnarray}
\frac{n (n-1) \cdots\overbrace{(n-k)}\cdots(n-h)}{k!(h-k)!} &=& \frac{n! (n-k-1)!}{k! (n-k)! (n-h-1)! (h-k)!} \nonumber \\[0.7pc]
&=& {n \choose k} {n-k-1 \choose h-k} \; . \nonumber
\end{eqnarray}
Therefore Equality (\ref{fla combinatoria 2}) can be written in the form
\begin{equation}\label{fla combinatoria 2 bis}
\sum_{j=k}^{h} (-1)^{j-k} {n \choose j} {j \choose k} = (-1)^{h-k}  {n \choose k} {n-k-1 \choose h-k} \; .
\end{equation}
Analogously, Equality (\ref{fla combinatoria 3}) is the same that
\begin{equation}\label{fla combinatoria 3 bis}
\sum_{k=0}^{h} (-1)^{h-k} {n \choose k} {n-k-1  \choose h-k} = 1  \; .
\end{equation}
}
\end{remark}

\;

\subsection{The general term of a sequence}

\ \par

\;

Now we give an expression for the general term of an arbitrary sequence in function of the difference operator.

\begin{teorema} \cite[Lemma 3.4]{bmmu} Let $a$ be a sequence in $G$. Then, for $n=0,1,2,3...$,
\begin{equation}\label{sequence a_n}
a_n = \sum_{k=0}^n  {n \choose k} D^k a_{0} \; .
\end{equation}
That is,
$$
a_n = \sum_{k=0}^n  {n \choose k}  \sum_{h=0}^k (-1)^{k-h} {k \choose h} a_{h} \; .
$$
\end{teorema}

\begin{proof} It is clear that Equality (\ref{sequence a_n}) holds for $n=0$ and $n=1$. Assume that (\ref{sequence a_n}) is true for  $0,1,...,n$ and we shall prove that it is also true for $n+1$. By  (\ref{diff op}) and the induction hypothesis
\begin{eqnarray}
a_{n+1} &=& D^{n+1} a_0-  \sum_{k=0}^n  (-1)^{n+1-k} {n+1 \choose k} a_k \nonumber \\
&=& D^{n+1} a_0 - \sum_{k=0}^n (-1)^{n+1-k} {n+1 \choose k} \sum_{j=0}^k {k \choose j} D^j a_{0} \; . \nonumber
\end{eqnarray}
From Lemma \ref{lema} we obtain
\begin{eqnarray}
a_{n+1} &=& D^{n+1} a_0 -  \sum_{j=0}^n  D^j a_0  \sum_{k=j}^n (-1)^{n+1-k} {n+1 \choose k} {k \choose j} \nonumber \\
&=& D^{n+1} a_0 - \sum_{j=0}^n  {n+1 \choose j} D^j a_0 \sum_{k=j}^n (-1)^{n+1-k} {n+1-j \choose k-j} \nonumber \\
&=& \sum_{j=0}^{n+1} {n+1 \choose j} D^j a_0 \; , \nonumber
\end{eqnarray}
because
\begin{eqnarray}
\sum_{k=j}^n (-1)^{n+1-k} {n+1-j \choose k-j} &=& \sum_{h=0}^{n-j} (-1)^{n-j+1-h} {n-j+1 \choose h} \nonumber \\[0.7pc]
&=& (-1)^{n-j+1} \sum_{h = 0}^{n-j} (-1)^{h} {n-j+1 \choose h} \nonumber \\[0.7pc]
&=& (-1)^{n-j+1} (-1)^{n-j} {n-j \choose n-j} \nonumber \\
 &=& -1 \nonumber
\end{eqnarray}
by Lemma \ref{lema combinatorio 1}.
\end{proof}

\;

\section{ARITHMETIC PROGRESSIONS}

\;

In this section $G$ denotes a commutative  additive group.

Now we give the central notion of this paper.

\begin{definition} Let $h$ be a non negative integer. A sequence $a = (a_n)_{n \geq0}$ in $G$ is called an {\rm arithmetic progression of order $h$} if $D^h a$ is constant, so $D^{h+1} a = 0$.
\end{definition}

\;

\subsection{The term general of an arithmetic progression}

\ \par

\;

In the next result we give some expressions for the general term  of an arithmetic progression.

\begin{teorema}\label{a p general term} Let $a$ be a sequence in $G$. The following assertions are equivalent:

\begin{enumerate}

\item $a$ is an arithmetic progression of order $h$.

\item For all $n \geq 0$,
\begin{equation}\label{term gral 1}
a_n = \sum_{k=0}^h {n \choose k} D^k a_0 \; .
\end{equation}

\item (\cite[Theorem 2.1]{bmm}, \cite[Proposition 2.2]{bmno}) For all $n \geq 0$,
\begin{equation}\label{term gral 2}
a_n = \sum_{k=0}^h (-1)^{h-k} \frac{n (n-1) \cdots\overbrace{(n-k)}\cdots(n-h)}{k!(h-k)!} a_k = \sum_{k=0}^h (-1)^{h-k} {n \choose k} {n - k -1 \choose h-k}  a_k \; .
\end{equation}

\item \cite[Proposition 3.7]{hms} For all $n \geq 0$,
\begin{equation}\label{term gral 4}
a_n = \displaystyle{ \frac{\displaystyle{\sum_{k=0}^h  (-1)^{h-k} {h \choose k} \frac{1}{n-k} a_k}}{\displaystyle{\sum_{k=0}^h (-1)^{h-k} {h \choose k} \frac{1}{n-k}}} } \; .
\end{equation}

\item \cite[Lemma 2.2]{duggal+muller} There is a polynomial $p_a$, with coefficients in $G$ of degree less or equal to $h$, such that $p_a (n) = a_n$, for any $n \geq 0$; that is, there are $\ga_{h}, \ga_{h-1}, ..., \ga_{1}, \ga_{0}$ in $G$ such that
\begin{equation}\label{term gral 3}
a_n = \ga_{h} n^{h} + \ga_{h-1} n^{h-1} + \cdots + \ga_{2} n^{2} + \ga_{1} n + \ga_0 \; ,
\end{equation}
where $\ga_h = \frac{1}{h!} D^h a_0$ and $\ga_0 = a_0$.
\end{enumerate}

\end{teorema}

\begin{proof} (1) $\Longrightarrow$ (2) By definition, $D^k a_0 = 0$, for any $k \geq h+1$. Applying (\ref{sequence a_n}) we obtain the expression of (2).

(2) $\Longrightarrow$ (3) By (\ref{diff op})  and (\ref{term gral 2}) we have
\begin{eqnarray}
a_n &=& \sum_{k=0}^h {n \choose k} D^k a_0 \nonumber \\
&=& \sum_{k=0}^h {n \choose k} \sum_{i=0}^k (-1)^{k-i} {k \choose i} a_i \nonumber \\
&=& \sum_{k=0}^h \sum_{i=0}^k (-1)^{k-i}  {n \choose k}  {k \choose i} a_i \; . \nonumber
\end{eqnarray}
Applying Lemmas \ref{lema} and \ref{lema combinatorio 2},
\begin{eqnarray}
a_n &=& \sum_{i=0}^h a_i \sum_{k=i}^h  (-1)^{k-i}  {n \choose k}  {k \choose i}  \nonumber \\
&=&  \sum_{i=0}^h  (-1)^{h-i} \frac{n (n-1) \cdots\overbrace{(n-i)}\cdots(n-h)}{i!(h-i)!} a_i \; . \nonumber
\end{eqnarray}

(3) $\Longleftrightarrow $ (4) It is enough to prove that
$$
\frac{n (n-1) \cdots\overbrace{(n-k)}\cdots(n-h)}{k!(h-k)!} = \displaystyle{\frac{\displaystyle{{h \choose k} \frac{1}{n-k}}}{\displaystyle{\sum_{i=0}^h (-1)^{h-i} {h \choose i} \frac{1}{n-i}}}} \; ;
$$
equivalently,
\begin{eqnarray}
h! &=& n(n-1) \cdots (n-h) \sum_{i=0}^h (-1)^{h-i} {h \choose i} \frac{1}{n-i} \nonumber \\
&=& h!  \sum_{i=0}^h (-1)^{h-i} \frac{n (n-1) \cdots\overbrace{(n-i)}\cdots(n-h)}{i!(h-i)!} \; . \nonumber
\end{eqnarray}
This is true by Lemma \ref{lema combinatorio 3}.

(3) $\Longrightarrow$ (5) Note that
$$
\frac{n (n-1) \cdots\overbrace{(n-i)}\cdots(n-h)}{k!(h-k)!}
$$
is a polynomial in $n$ of degree $h$, hence $a_n$ is a polynomial $p_a (n)$ in $n$ of degree less or equal to $h$; consequently
$$
a_n = \ga_{h} n^{h} + \ga_{h-1} n^{h-1} + \cdots + \ga_{2} n^{2} + \ga_{1} n + \ga_0 \; ,
$$
for some $\ga_{h}, \ga_{h-1}, ..., \ga_{1}, \ga_{0}$ in $G$.

It is clear that $a_0 = p_a(0) = \ga_0$. From  (\ref{term gral 2}) we obtain that the coefficient of $n^h$ is
\begin{eqnarray}
\sum_{k=0}^h (-1)^{h-k} \frac{1}{k! (h-k)!} a_k &=& \frac{1}{h!} \sum_{k=0}^h (-1)^{h-k} {h \choose k} a_k \nonumber \\
&=& \frac{1}{h!} D^h a_0 \; . \nonumber
\end{eqnarray}

(5) $\Longrightarrow$ (1) If $a_n$ is given by  (\ref{term gral 3}), we have
\begin{eqnarray}
D a_n &=& a_{n+1} - a_n \nonumber \\
&=& \sum_{k=0}^h \ga_k (n+1)^k - \sum_{k=0}^h \ga_k n^k  \nonumber \\
&=& \sum_{k=0}^h \ga_k \sum_{i=0}^{k-1} {k \choose i} n^i   \nonumber \\
&=& \sum_{k=1}^h \eta_{k-1} n^{k-1} \nonumber \\
&=& \sum_{k=0}^{h-1} \eta_{k} n^{k}  \; , \nonumber
\end{eqnarray}
for certain $\eta_i$ in $G$ ($i=0,...,h-1$). So, $D^h a_n$ is a polynomial in $n$ of degree less or equal to $1$. Therefore $D^{h+1} a=0$, which proves the statement.
\end{proof}

In view of the above theorem, the arithmetic progressions are also called {\it polynomial sequences.}

\begin{remark} {\rm In the proof of Proposition 3.7 in \cite{hms}, the authors observe that the polynomial $p_a$ associated to the arithmetic progression $a$ of order $h$ is the Lagrange polynomial of degree less or equal to $h$ which interpolates $(0,a_0), (1,a_1),...,(h,a_h)$. Moreover, by the normal form of the Lagrange polynomial it is obtained  (\ref{term gral 2}), using the Newton form yields (\ref{term gral 1})  and using the barycenter form yields (\ref{term gral 4}).
}
\end{remark}

\begin{definition} A sequence $a = (a_n)_{n \geq0}$ in $G$ is called an {\rm arithmetic progression of strict order $h$} if it is an arithmetic progression of order $h$ with $h=0$ or if it is an arithmetic progression of order $h \geq 1$, but is not of order $h-1$.
\end{definition}

\begin{teorema}\label{a p strict} Let $a$ be a sequence in $G$. The following assertions are equivalent:
\begin{enumerate}

\item $a$ is an arithmetic progression of strict order $h$.

\item For all $n \geq 0$,
\begin{equation}\label{term gral 1 strict}
a_n = \sum_{k=0}^h {n \choose k} D^k a_0 \quad \mbox{ and } \quad D^h a_0 \neq 0 \; .
\end{equation}

\item For all $n \geq 0$,
\begin{equation}\label{term gral 2 strict}
a_n = \sum_{k=0}^h (-1)^{h-k} \frac{n (n-1) \cdots\overbrace{(n-k)}\cdots(n-h)}{k!(h-k)!} a_k = \sum_{k=0}^h (-1)^{h-k} {n \choose k} {n - k -1 \choose h-k}  a_k \quad \mbox{ and } \quad D^h a_0 \neq 0 \; .
\end{equation}

\item For all $n \geq 0$,
\begin{equation}\label{term gral 4 strict}
a_n = \displaystyle{  \frac{ \displaystyle{ \sum_{k=0}^h (-1)^{h-k} {h \choose k} \frac{1}{n-k} a_k}}{\displaystyle{ \sum_{k=0}^h (-1)^{h-k} {h \choose k} \frac{1}{n-k}}} } \quad \mbox{ and } D^h a_0 \neq 0 \; .
\end{equation}

\item There is a polynomial $p_a$, with coefficients in $G$ of degree exactly $h$, such that $p_a (n) = a_n$, for any $n \geq 0$; that is, there are $\ga_{h} \neq 0, \ga_{h-1}, ..., \ga_{1}, \ga_{0}$ in $G$ such that
\begin{equation}\label{term gral 3 strict}
a_n = \ga_{h} n^{h} + \ga_{h-1} n^{h-1} + \cdots + \ga_{2} n^{2} + \ga_{1} n + \ga_0 \; ,
\end{equation}
being $\ga_h = \frac{1}{h!} D^h a_0 \neq 0$ and $\ga_0 = a_0$.
\end{enumerate}
\end{teorema}

\begin{proof} Taking into account Theorem \ref{a p general term} it is enough to prove that (1) $\Longrightarrow$ (5) because the other implications are clear. If the polynomial $p_a$  has degree less than $h$, then the sequence $a$ is an arithmetic progression of order $h-1$, hence it is not an arithmetic progression of strict order $h$.
\end{proof}

\;

\subsection{Subsequences of an arithmetic progression}

\ \par

\;

Let $a = (a_n)_{n \geq 0}$ be an arithmetic progressions of strict order $h \geq 0$. Two simple results, but very useful, are the following:
\begin{enumerate}
\item Fixed an integer $k \geq 1$, the subsequence $(a_{n+k})_{n \geq 0}$ is also an arithmetic progression of strict order $h$.
\item The sequence $A = (A_n)_{n \geq 0}$, defined by
$$
A_n = a_0 + a_1 + \cdots + a_{n-1} + a_n \; ,
$$
is an arithmetic  progression of strict  order $h+1$ since $DA = (a_{n+1})_{n \geq 0}$ is an arithmetic  progression of strict  order $h$.
\end{enumerate}

Let $b = (b_n)_{n \geq 0}$ be a subsequence of the sequence $a = (a_n)_{n \geq 0}$; that is, there exists a strictly increasing sequence  $g = (g_n)_{n \geq 0}$ of nonnegative integers such that
$$
b_n = a_{g_n} \; .
$$
The subsequence $b$ of $a$ is characterized by the sequence of steps $s = (s_n)_{n \geq 0}$:
$$
s_0 = g_0 \quad , \quad s_n = g_n - g_{n-1} \; (n \geq 1) \; .
$$
Notice that
\begin{equation} \label{g n}
g_n = s_0 + s_1 + \cdots + s_{n-1} + s_n \; .
\end{equation}

\begin{proposicion} Let $a = (a_n)_{n \geq 0}$ be an arithmetic progression of strict order $h \geq 0$ and let $b = (b_n)_{n \geq 0}$ be a subsequence of $a$ with sequence of steps $s = (s_n)_{n \geq 0}$, which is an arithmetic progression of strict order $k$. Then $b$ is an arithmetic progression of strict order $h(k + 1)$.
\end{proposicion}

\begin{proof} Let $p_a$ be the polynomial associated to $a$ which has an expression as (\ref{term gral 3 strict}). As $s$ is an arithmetic progression of strict order $k$, we have that the sequence $g$ given by (\ref{g n}) is an arithmetic progressions of strict order $k + 1$, hence its associated polynomial  $p_g$ have degree $k + 1$:
$$
  g_n = p_g (n) = \sigma_{k + 1} n^{k + 1} + \sigma_{k} n^{k} + \cdots + \sigma_1 n + \sigma_0 \quad (\sigma_{k + 1} \neq 0) \; .
$$
Consequently
\begin{eqnarray}
b_n &=& a_{g_n} \nonumber \\
&=& p_a (g_n) = \alpha_h g_n^h + \alpha_{h-1} g_n^{h-1} + \cdots + \alpha_1 g_n + \alpha_0 \nonumber \\
&=& \alpha_h p_g(n)^h + \alpha_{h-1} p_g(n)^{h-1} + \cdots + \alpha_1 p_g(n) + \alpha_0 \nonumber \\
&=& \alpha_h \sigma_{k + 1}^h n^{h (k + 1)} + \cdots \; , \nonumber
\end{eqnarray}
which is a polynomial in $n$ of degree  $h(k + 1)$. Therefore $b$ is an arithmetic progression of strict order $h (k + 1)$.
\end{proof}

\begin{corollary}\label{cor} Let $a = (a_n)_{n \geq 0}$ be an arithmetic progressions of strict order $h \geq 0$ and let $b = (b_n)_{n \geq 0}$ be the subsequence of $a$ defined by
$$
b_n = a_{dn} \quad (n \geq 0) \; ,
$$
for certain integer $d \geq 1$. Then $b$ is an arithmetic progression of strict order $h$.
\end{corollary}

\begin{proof} The result is obtained from the above proposition taking as sequence of steps $s = (s_n)_{n \geq 0}$ the constant sequence $s=(d)_{n \geq 0}$, which is an arithmetic progression of order $0$.
\end{proof}

\;

\subsection{Sequences of arithmetic progressions}

\ \par

\;

Now we work with a double sequence $(a_{i,j})_{i,j \geq 0}$ such that the files $(a_{i,j})_{j \geq 0}$ and the columns $(a_{i,j})_{i \geq 0}$ are arithmetic progressions. We obtain that the diagonal sequence $(a_{i,i})_{i \geq 0}$  is also an arithmetic progression. The results of this subsection are basically in \cite{bmno} and \cite{duggal+muller}.

First we give a necessary lemma.

\begin{lema}\label{factorialg ap}
Let $(a_i)_{i \geq 0}$ be an arithmetic progression of order $h$ and let $n$ be a positive integer with $n> 1$. Then
\begin{equation}\label{factorial ap}
\sum_{i=0}^{h+n}{h+n \choose i} (-1)^{h+n+1-i} i (i-1) \cdots (i-\ell) a_i  \;
\end{equation}
is zero, for  any $\ell \in \{0,1, ...,n-2\}$.
\end{lema}

\begin{proof}
Fixed  $\ell \in \{ 0, 1, ... , n-2\}$. The equation (\ref{factorial ap}) is equivalent to
\begin{equation}\label{e-lema ap}
\sum_{i=\ell+1}^{h+n} {h+n \choose i} (-1)^{h+n+1-i} i (i-1) \cdots (i-\ell) a_i  \; .
\end{equation}
Note that
\begin{eqnarray}
{h+n \choose i} i(i-1) \cdots (i-\ell) &=& \frac{(h+n)!}{i!(h+n-i)!} i (i-1) \cdots (i-\ell)\nonumber \\[1pc]
& =& \frac{(h+n) \cdots (h+n-\ell)(h+n-\ell-1)!}{(i-\ell-1)!(h+n-i)!}\nonumber \\[1pc]
&=&{h+n-\ell -1 \choose i-\ell-1} \prod_{j=1}^{\ell+1} (h+n-j+1) \;. \nonumber
\end{eqnarray}
Then  (\ref{e-lema ap}) agrees with
\begin{eqnarray}
\prod_{j=1}^{\ell+1} (h+n-j+1) \sum_{i=\ell+1}^{h+n} {h+n-\ell-1 \choose i-\ell-1} (-1)^{h+n-i+1} a_i  \nonumber \\
= \prod_{j=1}^{\ell+1} (h+n-j+1)\sum_{i=0}^{h+n-\ell-1}{h+n-\ell-1\choose i} (-1)^{h+n-\ell-i} a_{\ell+1+i} \; . \nonumber\\
\end{eqnarray}
As $(a_i)_{i \geq 0}$ is an arithmetic progression of order $h$, we have that  $(a_i)_{i \geq 0}$ is  an arithmetic progression of order
$h+n-\ell-2$ for all $\ell\in \{0, 1, ..., n-2\}$; that is,
$$
\sum_{i=0}^{h+n-\ell-1} (-1)^{h+n-\ell-1-i} {h+n-\ell-1 \choose i}  a_i = 0 \; .
$$
Hence  we obtain the result.
\end{proof}

\begin{teorema}\label{product ap} \cite[Corollary 2.5]{duggal+muller}.
Let $(a_{i,j})_{i,j \geq 0}$ be a double sequence such that the sequence $(a_{i,j})_{j \geq 0}$ is an  arithmetic progression of order $k \geq 0$, for all $i \geq 0$, and the sequence $(a_{i,j})_{i \geq 0}$ is an arithmetic progression of order $h \geq 0$, for all $j \geq 0$. Then the diagonal sequence $(a_{i,i})_{i \geq 0}$ is an  arithmetic progression of order $k+h$.
\end{teorema}

\begin{proof}
For $k=0$ is clear. Assume that $k>0$. Denote
$$
A :=\sum_{r=0}^{k+h+1}{k+h+1 \choose r}(-1)^{k+h+1-r} a_{r,r} \; .
$$
Let us prove that $A = 0$. From part (3) of Theorem \ref{a p general term} we obtain
\begin{eqnarray}
A &=& \sum_{r=0}^{k+h+1} {k+h+1 \choose r} (-1)^{k+h+1-r} \left( \sum_{i=0}^{k} g(r,i,k) a_{r,i} \right) \;\nonumber \\
 &=&\sum_{i=0}^{k} \sum_{r=0}^{k+h+1} {k+h+1 \choose r} (-1)^{k+h+1-r}  g(r,i,k) a_{r,i} \; ,\nonumber
\end{eqnarray}
where
$$
g(r,i,k) := (-1)^{k-i} \frac{r (r-1)  \cdots \overbrace{(r-i)} \cdots (r-k)}{i! (k-i)!} \;.
$$
By elementary properties of polynomials we have that there exist scalars $(\eta_{i,j})_{j=0}^{k}$ such that
\begin{equation}\label{polinomio ap}
r (r-1) \cdots \overbrace{(r-i)} \cdots (r-k)= \eta_{i,0}+ \eta_{i,1} r + \eta_{i,2} r(r-1) + \cdots  + \eta_{i,k} r(r-1) \cdots (r-k+1)
\end{equation}
for $i \in \{ 0, 1, \cdots, k\}$.  By equality (\ref{polinomio ap}) and Lemma \ref{factorialg ap} we obtain the result.
\end{proof}

\;

\subsection{Perturbation of arithmetic progressions by nilpotents on a ring}

\ \par

\;

In this subsection the setting is a ring $R$ with unity $e$, whose operations are denoted by $+$ and $\cdot$, being $\cdot$ not necessarily commutative. The main result is included in \cite{bmmn 2}.

Given an integer $n \geq 1$, an element $a \in R$ is called {\it $n$-nilpotent} if $a^n =0$ but $a^{n-1} \neq 0$.

In the next theorem we consider a sequence $(y^k x^k)_{k \geq 0}$ in $R$ which is an arithmetic progression of strict order $h$. By part (5) of Theorem \ref{a p strict}, for every $k \geq 0$,
\begin{equation}\label{pa ring}
y^k x^k = \sum_{\ell=0}^{h} c_\ell k^\ell \quad (c_h \neq 0) \; ,
\end{equation}
where $c_\ell \in R$ ($0 \leq \ell \leq h$) do not depend on $k$.

\begin{teorema}\label{ring perturbation} Let $R$ be a ring. Let $x,y \in R$ such that $(y^k x^k)_{k \geq 0}$ is an arithmetic progression of strict order $h$. Let $a,b \in R$ such that $a$ is $n$-nilpotent, $b$ is $m$-nilpotent, $ax=xa$ and $by=yb$. Then the sequence  $((y+b)^k (x+a)^k)_{k \geq 0}$ is an arithmetic progression of order $n+m+h-2$. Moreover, it is of strict order $n+m+h-2$ whenever $b^{m-1} y^{n-m} c_h a^{n-1} \neq 0$, if $m \leq n$, or whenever  $b^{m-1} c_h x^{m-n} a^{n-1}  \neq 0$, if $m > n$.
\end{teorema}

\begin{proof} Denote by $c \wedge d$ the minimum of $c$ and $d$. We have, for every $k \geq 0$,
$$
(y+b)^{k} (x+a)^k = \left( \sum_{i=0}^{k} {k \choose i} b^{i} y^{k-i} \right) \left( \sum_{j=0}^{k} {k \choose j}  x^{ k-j} a^j \right) \; .
$$
Consequently, for every $k \geq 0$,
\begin{equation}\label{+}
(y+b)^{k} (x+a)^k  = \sum_{i=0}^{k \wedge (m-1)} \sum_{j=0}^{k \wedge (n-1)} {k \choose i} {k \choose j} b^{i} y^{k-i}  x^{k-j} a^j \; .
\end{equation}
Consider two cases.

{\it Case 1. $m \leq n$.} Taking into account  (\ref{pa ring}),  then (\ref{+}) can be written in the following way:
\begin{eqnarray}
(y+b)^{k} (x+a)^k &=& \sum_{i=0}^{k \wedge (m-1)} \sum_{j=i}^{k \wedge (n-1)} {k \choose i} {k \choose j} b^{i} y^{j-i} y^{k-j}  x^{k-j} a^j  \nonumber \\
&+& \sum_{i=1}^{k \wedge (m-1)} \sum_{j = 0}^{i-1} {k \choose i} {k \choose j} b^{i} y^{k-i}   x^{k-i} x^{i-j} a^j \nonumber \\
&=& \sum_{i=0}^{k \wedge (m-1)} \sum_{j=i}^{k \wedge (n-1)} {k \choose i} {k \choose j} b^{i} y^{j-i} \left( \sum_{\ell = 0}^h c_\ell (k-j)^\ell \right) a^j \nonumber \\
&+& \sum_{i=1}^{k \wedge (m-1)} \sum_{j = 0}^{i-1} {k \choose i} {k \choose j} b^{i}  \left( \sum_{\ell = 0}^h c_\ell (k-i)^\ell \right) x^{i-j} a^j \nonumber \\
&=& \sum_{i=0}^{k \wedge (m-1)} \sum_{j = i}^{k \wedge (n-1)} \sum_{\ell = 0}^h {k \choose i} {k \choose j} (k-j)^\ell b^{i} y^{j-i} c_\ell a^j \nonumber \\
&+& \sum_{i=1}^{k \wedge (m-1)} \sum_{j = 0}^{i-1} \sum_{\ell = 0}^h  {k \choose i} {k \choose j} (k-i)^\ell  b^{i} c_\ell x^{i-j} a^j \; . \nonumber
\end{eqnarray}
We write:
$$
p_{i,j,\ell} (k) := {k \choose i} {k \choose j} (k-j)^\ell \quad , \quad u_{i,j,\ell} := b^{i} y^{j-i} c_\ell a^j \;  ,
$$
being $0 \leq i \leq k \wedge (m-1)$, $i \leq j \leq k \wedge (n-1)$ and $0 \leq \ell \leq h$; analogously,
$$
q_{i,j,\ell} (k) := {k \choose i} {k \choose j} (k-i)^\ell \quad , \quad v_{i,j,\ell} := b^{i}  c_\ell x^{i-j} a^j \; ,
$$
for $0 \leq j < i \leq k \wedge (m-1)$ and $0 \leq \ell \leq h$. Hence
$$
(y+b)^{k} (x+a)^k  = \sum_{i=0}^{k \wedge (m-1)} \sum_{i \leq j}^{k \wedge (n-1)} \sum_{\ell = 0}^h p_{i,j,\ell} (k) u_{i,j,\ell} + \sum_{i=1}^{k \wedge (m-1)} \sum_{j = 0}^{i-1} \sum_{\ell = 0}^h  q_{i,j,\ell} (k) v_{i,j,\ell} \; .
$$
Note that the polynomial $p_{i,j,\ell}(k)$ has degree $i+j+\ell$, which is maximum with value $m+n+h-2$; analogously, the degree of $q_{i,j,\ell}(k)$ is $i+j+\ell$, which is maximum with value $m+m+h-3$. Therefore $(y+b)^k (x+a)^k$ is a polynomial of degree less or equal to $m+n+h-2$. Actually  the degree is $m+n+h-2$ whenever the coefficient of $k^{m+n+h-2}$ is non-null; that is,
$$
u_{m-1,n-1,h} = b^{m-1} y^{n-m} c_h a^{n-1}  \neq 0 \; .
$$

{\it Case 2. $m>n$.} Analogously to Case 1, (\ref{+}) can be written
\begin{eqnarray}
(y+b)^{k} (x+a)^k &=& \sum_{i=0}^{k \wedge (n-1)} \sum_{j = i}^{k \wedge (n-1)} {k \choose i} {k \choose j} b^{i} y^{j-i} y^{k-j} x^{k-j} a^j \nonumber \\
&+& \sum_{j=0}^{k \wedge (n-1)} \sum_{i = j + 1}^{k \wedge (m-1)} {k \choose i} {k \choose j} b^{i} y^{k-i}  x^{k-i} x^{i-j} a^j \nonumber \\
&=& \sum_{i=0}^{k \wedge (n-1)} \sum_{j=i}^{k \wedge (n-1)} \sum_{\ell = 0}^h {k \choose i} {k \choose j} (k-j)^\ell b^{i} y^{j-i} c_\ell a^j \nonumber \\
&+& \sum_{j=0}^{k \wedge (n-1)} \sum_{i = j +1}^{k \wedge (m-1)} \sum_{\ell = 0}^h  {k \choose i} {k \choose j} (k-i)^\ell  b^{i} c_\ell x^{i-j} a^j \; . \nonumber
\end{eqnarray}
We write:
$$
p_{i,j,\ell} (k) := {k \choose i} {k \choose j} (k-j)^\ell \quad , \quad u_{i,j,\ell} := b^{i} y^{j-i} c_\ell a^j \;  ,
$$
being $0 \leq i \leq j \leq k \wedge (n-1)$ and $0 \leq \ell \leq h$; similarly,
$$
q_{i,j,\ell} (k) := {k \choose i} {k \choose j} (k-i)^\ell \quad , \quad v_{i,j,\ell} := b^{i} c_\ell x^{i-j} a^j \;   ,
$$
being $0 \leq j \leq k \wedge (n-1)$, $j+1 \leq i \leq k \wedge (m-1)$ and $0 \leq \ell \leq h$. Then
$$
(y+b)^{k} (x+a)^k  = \sum_{i=0}^{k \wedge (n-1)} \sum_{j=i}^{k \wedge (n-1)} \sum_{\ell = 0}^h p_{i,j,\ell} (k) u_{i,j,\ell} + \sum_{j=0}^{k \wedge (n - 1)} \sum_{i = j + 1}^{k \wedge (m-1)} \sum_{\ell = 0}^h  q_{i,j,\ell} (k) v_{i,j,\ell} \; .
$$
The degree of $p_{i,j,\ell}(k)$ is $i+j+\ell$, which is maximum with value $n+n+h-2$, and the degree of $q_{i,j,\ell}(k)$ is $i+j+\ell$, with maximum $m+n+h-2$. Hence $(y+b)^k (x+a)^k$ is a polynomial of degree less or equal to $m+n+h-2$. Really the degree is $m+n+h-2$ if the coefficient of $k^{m+n+h-2}$ is non-null; that is,
$$
v_{m-1,n-1,h} = b^{m-1} c_h x^{n-m} a^{n-1}  \neq 0 \; .
$$
Thus the proof is finished.
\end{proof}

\begin{corollary}  Let $R$ be a ring. Let $x \in R$ such that $(x^k)_{k \geq 0}$ is an arithmetic progression of strict order $h$ and let $a \in R$ be an $n$-nilpotent such that $ax=xa$. Then the sequence  $((x+a)^k)_{k \geq 0}$ is an arithmetic progression of order $n+h-1$; moreover, it is of strict order $n+h-1$ if $a^{n-1}c_h x^{n-1}\neq 0$, where
$$
x^k= \sum_{\ell = 0}^h c_{\ell} k^{\ell} \; .
$$
\end{corollary}

\;

\;

\section{NUMERICAL ARITHMETIC PROGRESSIONS}

\;

In this section we consider sequences of scalars, mainly sequences of positive numbers.

Throughout this section $h,k, \ell$ are integers and $q >0$ a real number.

\;

\subsection{Recursive equations}

\ \par

\;

The contain of this subsection is taken from  \cite[p. 251]{bdm};  see also \cite[p. 104]{agarwal} and \cite[Theorem 3.7]{k+p}.

The sequence $(a_n)_{n \geq 0}$ of scalars is an arithmetic progression of strict order $h$ if and only if, for every $n=0,1,2 ...$,
\begin{equation}\label{recursive 1}
D^{h+1} a_n = \sum_{k=0}^{h+1} (-1)^{h+1-k}  {h+1 \choose k} a_{n+k} =0 \quad \mbox{ and} \quad D^{h} a_n \neq 0 \; .
\end{equation}

Notice that (\ref{recursive 1}) is a recursive equation. Let us  introduce some classical results to solve this  type of
equations. We are interested in sequences $(y_n)_{n \geq 0}$
which verify the recursive equation
\begin{equation}\label{recursive 2}
y_{h+1+n} + \gamma_{h} y_{h+n} + \gamma_{h-1} y_{h-1+n} + \cdots + \gamma_{1} y_{1+n} + \gamma_{0} y_{n} = 0 \; ,
\end{equation}
for certain $h \geq 0$ and any $n \geq 0$, being $\gamma_i$ complex numbers ($0 \leq i \leq h$). The {\it characteristic polynomial}
of the equation (\ref{recursive 2}) is given by
$$
q(z) = z^{h+1} + \gamma_{h} z^{h} + \gamma_{h-1} z^{h-1} + \cdots + \gamma_1 z + \gamma_0 \; ,
$$
which can be written in the form
\begin{equation}\label{polynomial 1}
q(z) = (z-z_1)^{h_1} (z-z_2)^{h_2} \cdots (z-z_k)^{h_k}  \; ,
\end{equation}
where $h_1 + h_2 + \cdots + h_k = h+1$ and $z_i \neq z_g$ for $i \neq
g$. It is well known (see for example \cite[Theorem 3.7]{k+p} and
\cite[page 104]{agarwal}) that the set of all complex sequences
which  verify (\ref{recursive 2}) is a vectorial subspace of the
space $\mathbb{C}^\mathbb{N}$ of all complex sequences, it has
dimension $h+1$ and a basis is formed by the sequences
\begin{eqnarray}\label{basis 1}
(z_1^n)_{n \geq 0} , (nz_1^n)_{n \geq 0} , (n^2z_1^n)_{n \geq 0} , ... , (n^{h_1-1}z_1^n)_{n \geq 0} \, ,\nonumber  \\
(z_2^n)_{n \geq 0} , (nz_2^n)_{n \geq 0} , (n^2 z_2^n)_{n \geq 0} , ... , (n^{h_2-1} z_2^n)_{n \geq 0} \, ,\nonumber  \\
... ... ... ... ... ... ... ... ... ... ... ... ... ... ... ... ... ... ... ... ... ... ..\,\,\,\,\, ,\\
(z_j^n)_{n \geq 0} , (n z_j^n)_{n \geq 0} , (n^2 z_j^n)_{n \geq 0} , ... ,  (n^{h_j-1} z_j^n)_{n \geq 0} \, .\nonumber
\end{eqnarray}
Therefore there exists an identification between the recursive equation
(\ref{recursive 2}),  the characteristic polynomial (\ref{polynomial
1}), the subspace of sequences which satisfy the recursive equation
and its basis (\ref{basis 1}).

Consider now an arithmetic progression  $a = (a_n)_{n \geq 0}$ of order $h$  which verifies (\ref{recursive 1}), for every $n \geq 0$. The characteristic polynomial associated is $q_1(z) = (z-1)^{h+1}$, hence the sequence
$(a_n)_{n \geq 0}$ is a linear combination of the sequences $(1)_{n\geq 0}$, $(n)_{n \geq 0}$, $(n^2)_{n \geq 0}$, ...., $(n^{h})_{n\geq 0}$. This is another way to obtain the expression (\ref{term gral 3}). Note that the sequence $a$ has strict order $h$ if $a$ is not a linear combination of only the sequences $(1)_{n\geq 0}$, $(n)_{n \geq 0}$, $(n^2)_{n \geq 0}$, ...., $(n^{h-1})_{n\geq 0}$.

\;

\subsection{Subsequences of numerical sequences}

\ \par

\;

In the next result we consider a sequence $(a_n)_{n \geq 0}$ such that its subsequences $(a_{cn})_{n \geq 0}$ and $(a_{dn})_{n \geq 0}$ are arithmetic progressions. This result is essentially \cite[Theorem 3.6]{bdm}.

\begin{teorema}\label{sub a p} Let $a = (a_n)_{n \geq 0}$ be a numerical sequence. Suppose that $(a_{cn})_{n \geq 0}$ is an arithmetic progression of strict order $h$ and $(a_{dn})_{n \geq 0}$ is an arithmetic progressions of strict order $k \geq 0$, for $c,d \geq 1$ and $h,k \geq 0$. Then $(a_{en})_{n \geq 0}$ is an arithmetic progression of strict order $\ell$, being $e$ the  greatest common divisor of $c$ and $d$, and $\ell$ the minimum of $h$ and $k$.
\end{teorema}

\begin{proof} As $(a_{cn})_{n \geq 0}$ is an arithmetic progression of strict order $h$ satisfies the recursive equation
\begin{equation}\label{r isometry}
 \displaystyle \sum_{i=0}^{h+1} (-1)^{h+1-i} {h+1 \choose i}  a_{c(i+j)} = 0 \; ,
\end{equation}
for all $j\geq 0$. The equation (\ref{r isometry}) has  the characteristic polynomial $p_1(z):=q_c(z)^{h+1} :=(z^c - 1)^{h+1}$. We call
$V_c$ to the subspace of $\mathbb{C}^\mathbb{N}$ formed by all complex sequences which verify (\ref{r isometry}). Then   $V_c$ has dimension $\dim V_c = (h+1)c$.

Analogously, $(a_{dn})_{n \geq 0}$ satisfies the equation
\begin{equation}\label{s isometry}
\displaystyle \sum_{i=0}^{k+1} (-1)^{k+1-i} {k+1 \choose i}  a_{d(i+j)} = 0 \; ,
\end{equation}
for all $j\geq 0$. The characteristic polynomial of (\ref{s isometry}) is $p_2(z) := q_d(z)^{k+1} = (z^d-1)^{k+1}$. Now the set of all the sequences which verify (\ref{s isometry}) is a vectorial subspace  $V_{d}$ of $\mathbb{C}^\mathbb{N}$ and  $\dim V_{d} = d(k+1)$.

Therefore the  sequences which verify both equations (\ref{r isometry}) and (\ref{s isometry}) are the sequences in the subspace
$V_c \cap V_d$, those characteristic polynomial is the greatest common divisor of $q_c(z)^{h+1}$  and $q_d(z)^{k+1}$, which is the polynomial $q_e(z)^{\ell+1} := (z^e-1)^{\ell+1}$,
where $e$ is the greatest common divisor of $c$ and $d$, and $\ell$ is the minimum of $h$ and $k$. Consequently $(a_{en})_{n \geq 0}$ is an arithmetic progression of strict order $\ell$.
\end{proof}

\;

\subsection{Monotony of positive arithmetic progressions}

\ \par

\;

We say that the sequence $a= (a_n)_{n \geq 0}$ of real numbers is {\it positive} if $a_n >0$ for any $n \geq 0$. This subsection is dedicated to positive arithmetic progressions.

\begin{proposicion}\label{monotony} Let $a$ be an arithmetic progression of order $h$ of positive real numbers. Then  the sequence $a$  is eventually increasing; that is, there is a positive integer $n_0$ such that, for any integer $n \geq n_0$, we have that $a_n \leq a_{n+1}$. Moroever, if  $a$ is not constant, then
$$
\lim_{n \rightarrow \infty} a_n = \infty \; .
$$
\end{proposicion}

\begin{proof} If $a$ is constant the result is clear. Assume that $a$ is not constant, so is an arithmetic progression of strict order $h \geq 1$. The sequence $(a_n)_{n \geq 0}$ satisfies
(\ref{term gral 3}) with positive leader coefficient $\g_{h}$, so $a_n \longrightarrow \infty$ as $n \longrightarrow \infty$ and it is eventually increasing.
\end{proof}

\begin{example} {\rm The sequence $a=(a_n)_{n \geq 0}$ of natural numbers, given by $a_n = n^2 - 2n + 2$ ($n \geq 0$) is a positive arithmetic progression of strict order $2$. The first terms of $a$ are $2,1,2,5,10...$, hence $a$ is not monotone.
}
\end{example}

\begin{example} {\rm The sequence of natural numbers $a=(a_n)_{n \geq 0} = (n^3 - 5n^2 +8n +1)_{n \geq 0}$ is a positive arithmetic progression of strict order $3$, is increasing, but is not strictly increasing because $a_1=a_2=5$: $a = (1,5,5,7,17... )$.
}
\end{example}

\;

\subsection{Powers of positive sequences}

\ \par

\;

Given a positive  sequence $a= (a_n)_{n \geq 0}$ and a real number $q>0$, we define the sequence $a^q := (a_n^q)_{n \geq 0}$.

\begin{proposicion}\label{prop r+q} Let $a= (a_n)_{n \geq 0}$ be a positive sequence, $q,r >0$ and $k,h \geq 0$. If $a^q$ is an arithmetic progression of strict order $k$ and if $a^r$ is an arithmetic progression of strict order $h$, then $a^{r+q}$  is an arithmetic progression of strict order $k+h$.
\end{proposicion}

\begin{proof} The polynomials associated to $a^q$ and $a^r$ are, respectively,
$$
a_n^q = \alpha_k n^k + \alpha_{k-1} n^{k-1} + \cdots + \alpha_0 \quad  ,  \quad a_n^r = \beta_h n^h + \beta_{h-1} n^{h-1} + \cdots + \beta_0 \quad (\alpha_k \neq 0 , \beta_h \neq 0) \; .
$$
Then
$$
a_n^{r+q} = \alpha_k \beta_h n^{k+h} +  \cdots + \alpha_0 \beta_0 \quad (\alpha_k  \beta_h \neq 0) \; .
$$
Therefore $a^{r+q}$  is an arithmetic progression of strict order $k+h$.
\end{proof}

\begin{proposicion}\label{prop rk=hq} \cite[Proposition 4.1]{hms} Let $a= (a_n)_{n \geq 0}$ be a positive sequence, $q,r >0$ and $k,h \geq 0$. If $a^q$ is an arithmetic progression of strict order $k$ and if $a^r$ is an arithmetic progression of strict order $h$, then
\begin{equation}\label{rk=hq}
rk = hq \; .
\end{equation}
\end{proposicion}

\begin{proof} Consider  the polynomials associated to the sequences $a^q$ and $a^r$, respectively: for all integers $n \geq 0$,
$$
a_n^q = \alpha_k n^k + \alpha_{k-1} n^{k-1} + \cdots + \alpha_1 n + \alpha_{0} \quad (\alpha_k \neq 0) \; ,
$$
$$
a_n^r = \beta_h n^h + \beta_{h-1} n^{h-1} + \cdots + \beta_1 n + \beta_{0} \quad (\beta_h \neq 0) \; .
$$
Note that $k=0 \Longleftrightarrow h=0$, hence (\ref{rk=hq}) holds. Assume $k,h \geq 1$. In this case there exist
$$
\lim_{n \rightarrow \infty} \frac{a_n^q}{n^k} = \alpha_k > 0 \quad \mbox{ and }  \quad \lim_{n \rightarrow \infty} \frac{a_n^r}{n^h} = \beta_h > 0 \; .
$$
Thus we have the equality
$$
\lim_{n \rightarrow \infty} \frac{a_n^{qr}}{n^{rk}} = \lim_{n \rightarrow \infty} \frac{a_n^{rq}}{n^{hq}} \; ,
$$
so $kr=hq$.
\end{proof}

Taking $k,h \geq 1$ in the statement of Proposition \ref{prop rk=hq}, hence $a$ is non-constant, it is possible to improve the result.

\begin{proposicion}\label{prop c d} \cite[Lemma 4.2]{hms} Let $a= (a_n)_{n \geq 0}$ be a non-constant positive sequence, $q,r >0$  and $k,h \geq 1$. If $a^q$ is an arithmetic progression of strict order $k$ and if $a^r$ is an arithmetic progression of strict order $h$, then $a^t$ is an arithmetic progression of strict order $d$, where $d = gcd \{k,h\}$ and
    \begin{equation}\label{gcd}
t = \frac{qd}{k} = \frac{rd}{h} \; .
\end{equation}
\end{proposicion}

\begin{proof} Let $P := p_{a^q}$ and $Q := p_{a^r}$ the polynomials associated to the sequences $a^q$ and $a^r$, respectively. For any integer $n \geq 0$,
$$
P (n) = a_n^q = \alpha_k n^k + \alpha_{k-1} n^{k-1} + \cdots + \alpha_1 n + \alpha_{0} \quad (\alpha_k \neq 0) \; ,
$$
$$
Q (n) = a_n^r = \beta_h n^h + \beta_{h-1} n^{h-1} + \cdots + \beta_1 n + \beta_{0} \quad (\beta_h \neq 0) \; .
$$
From Proposition \ref{prop rk=hq} we obtain $rk=hq$. Hence
$$
P(n)^h = a_n^{qh} = a_n^{kr} = Q(n)^k \; .
$$
Therefore $P^h = Q^k$, so $P$ and $Q$ have the same zeroes. If $z_i$ ($1 \leq i \leq \ell$) is a common zero with multiplicity $u_i$ in $P$ and multiplicity  $v_i$ in $Q$, then $u_ih=v_ik$. Moreover, if $d=gcd \{ k,h \}$, then we obtain that $k/d$ is divisor of $u_i$. Thus, for certain integers $w_i$ ($1 \leq i \leq \ell$), for all $n \geq 0$,
$$
P(n) = (n-z_1)^{u_1} ... (n-z_\ell)^{u_\ell} = \left( (n-z_1)^{w_1} ... (n-z_\ell)^{w_\ell}  \right)^{k/d} = R(n)^{k/d} \; ,
$$
where $R$ is a polynomial of degree $d$. Consequently
$$
R(n) = P(n)^{d/k} = a_n^{qd/k} = a_n^t  \; ,
$$
being $t = qd/k$. Hence $a^t$ is an arithmetic progression of strict order $d$. As $rk=hq$, we can write $t = rd/h$.
\end{proof}

\begin{proposicion}\label{prop 3 options} \cite[Proposition 4.3]{hms} Let $a= (a_n)_{n \geq 0}$ be a positive sequence. Then $a$ satisfies exactly one of the following assertions:
\begin{enumerate}
\item $a^q$ is not an arithmetic progression of order $h$, for all $q>0$  and $h \geq 0$.
\item $a$ is constant; hence $a^q$ is an arithmetic progression of order $h$, for all $q>0$ and $h \geq 0$.
\item There are unique integer $\ell \geq 1$ and real $s > 0$ such that, for every $k = 1,2,3...$, the sequence $a^{ks}$ is an arithmetic progression of strict order $k\ell$. Moreover, if $a^q$ is an arithmetic progression of strict order $h$, then $q=ks$ and $h=k\ell$, for some $k=1,2,3...$
\end{enumerate}
\end{proposicion}

\begin{proof} Suppose that (1) and (2) fail, hence $a$ is not constant. Let $\ell \geq 1$ the least integer such that $a^s$ is an arithmetic progression of strict order $\ell$ for some $s>0$. From Proposition \ref{prop rk=hq} we obtain that $s$ is unique. Let $P := p_{a^s}$ be the polynomial of degree $\ell$ associated to the sequence $a^s$. Then, for every $k \geq 1$, the polynomial $P^k$ has degree $k\ell$ and
$$
P (n)^k = a_n^{ks}  \; ,
$$
for $n \geq 0$. Thus $a^{ks}$ is an arithmetic progression of strict order  $k\ell$. Conversely, if $a^q$ is an arithmetic progression of strict order $h \geq 1$, then $a^t$ is an arithmetic progression of strict order $d$, where $d = gcd \{ \ell,h \}$ and $t = \frac{qd}{h} = \frac{sd}{\ell}$. From the minimality of $\ell$ we have that $d=\ell$ and then there is $k = \frac{q}{s} = \frac{h}{\ell} = 1,2,3...$ such that $h=k\ell$ and $q=ks$.
\end{proof}

The Proposition  \ref{prop 3 options} leads to following notion.

\begin{definition}\label{a p proper}   Let $s,q>0$ be real numbers and $\ell,h > 0$ be integers. Given a non-constant positive sequence $a$,  we say that $a$ is an {\rm arithmetic progression of proper order} $(s,\ell)$ whenever $a^s$ is an  arithmetic progression of strict order $\ell$ and if $a^q$ is an arithmetic progression of strict order $h$, then $s \leq q$ and $\ell \leq h$.
\end{definition}

Notice that, if some power $a^q$ of a positive sequence $a$ is an arithmetic progression, then there exist a real $s>0$  and an integer $\ell \geq 1$ such that $a$ is an arithmetic progression of proper order $(s,\ell)$. Moreover $a^{ks}$ is of strict order $k \ell$ and of order $h \geq k \ell$, for all integer $k \geq 1$ and any integer $h$.

\begin{center}
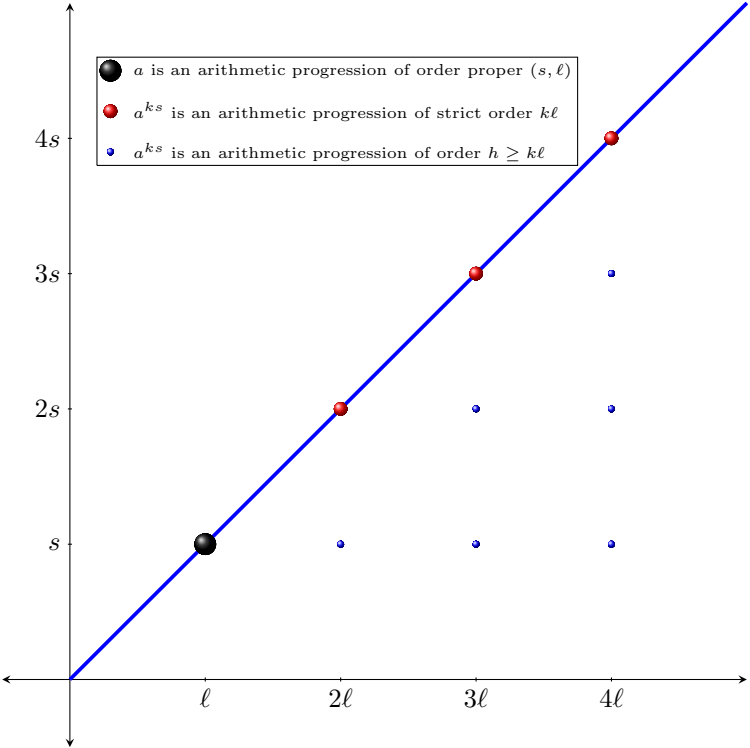
\begin{figure}[ht]
\begin{tikzpicture}[scale=0.18,>=stealth]
 \draw [<->] (0,-5)--(0,50);
 \draw [<->] (-5,0)--(50,0);
 \draw [blue,double=blue] plot [domain=0:50] (\x,\x);
\node[below] at (10,0) {$\ell $};
\draw (10,-0.15)--(10,0.15);
\node[below] at (20,0) {$2\ell $};
\draw (20,-0.15)--(20,0.15);
\node[below] at (30,0) {$3\ell $};
\draw (30,-0.15)--(30,0.15);
\node[below] at (40,0) {$4\ell $};
\draw (40,-0.15)--(40,0.15);
\node[left] at (0,10) {$s $};
\draw (-0.15,10)--(0.15,10);
\node[left] at (0,20) {$2s $};
\draw (-0.15,20)--(0.15,20);
\node[left] at (0,30) {$3s $};
\draw (-0.15,30)--(0.15,30);
\node[left] at (0,40) {$4s $};
\draw (-0.15,40)--(0.15,40);
\shade[shading=ball,ball color=red] (20,20) circle (.5);
\shade[shading=ball,ball color=red] (30,30) circle (.5);
\shade[shading=ball,ball color=red] (40,40) circle (.5);
\shade[shading=ball,ball color=blue] (20,10) circle (.25);
\shade[shading=ball,ball color=blue] (30,10) circle (.25);
\shade[shading=ball,ball color=blue] (40,10) circle (.25);
\shade[shading=ball,ball color=blue] (30,20) circle (.25);
\shade[shading=ball,ball color=blue] (40,20) circle (.25);
\shade[shading=ball,ball color=blue] (40,30) circle (.25);
\shade[shading=diamond,ball color=black] (10,10) circle (.8);
\shade[shading=diamond,ball color=black] (3,45) circle (.8);
\node[right] at (4,45) {\tiny $a$ is an arithmetic progression of order proper $(s,\ell)$};
\shade[shading=ball,ball color=red] (3,42) circle (.5);
\node[right] at (4,42) {\tiny  $a^{ks}$ is an arithmetic progression of strict order  $k\ell$};
\shade[shading=ball,ball color=blue] (3,39) circle (.25);
\node[right] at (4,39) {\tiny $a^{ks}$ is an arithmetic progression of order  $h\geq k\ell$};
\draw (2,38)  -- (37.5,38) -- (37.5,46) -- (2,46) -- (2,38);
 \end{tikzpicture}
 \caption{Graphical  interpretation of the order of a non-constant arithmetic progression}
 \end{figure}
\end{center}

\begin{remark}\label{pi} {\rm For a positive sequence $a$, in \cite{hms} the authors consider the sets
$$
\pi(a) :=  \{ (h,q) : h \geq 0 \;  q >  0 \; a^q \mbox{ is an arithmetic progression of order } h \} \; ,
$$
$$
\widehat{\pi} (a) :=  \{ (h,q) : h \geq 0 \; q >  0 \; a^q \mbox{ is an arithmetic progression of strict order } h \} \; .
$$
Proposition \ref{prop 3 options} can be formulated in the following form: the positive sequence $a$ verifies exactly one of the following:
\begin{enumerate}
\item $\pi(a) = \widehat{\pi} (a) = \emptyset$.
\item $\widehat{\pi} (a) = \{ 0 \} \times ]0,\infty[$ and $\pi(a) = \{ 0, 1, 2, 3... \} \times ]0,\infty[$
\item there are unique integer $\ell \geq 1$ and real $s > 0$ such that
$$
\widehat{\pi} (a) = \{ (k\ell,ks): k=1,2,3... \} \mbox{ and } \pi(a) = \{(h,ks) : k=1,2,3... ; h \geq k\ell\} \; .
$$
\end{enumerate}
}
\end{remark}

\;

\;

\section{APPLICATIONS TO $(m,q)$-ISOMETRIES}

\;

Now we apply the results of the previous sections to $(m,q)$-isometries. Consider three settings: metric spaces, Banach spaces and Hilbert spaces.

In this section $q,r,s$ denote real numbers and $m,n,h,k,\ell$ are integers.

\;

\subsection{$(m,q)$-isometries on metric spaces}

\ \par

\;

Throughout this section, $E$ denotes a metric space, $d$ its distance and $T: E \longrightarrow E$ a map.

\begin{definicion} A map $T : E \longrightarrow E$  is called an {\rm $(m,q)$-isometry}  if, for all $x,y \in E$,
\begin{equation}\label{isometria ms}
\sum_{k=0}^m (-1)^{m-k} {m\choose k}  d(T^k x, T^{k}y)^q = 0 \; .
\end{equation}
\end{definicion}

It is clear that $(1,q)$-isometries are isometries.

\begin{proposition}\label{en em} Let $T : E \longrightarrow E$ be a map. Fixed $m \geq 1$ and $q>0$, the following assertions are equivalent:
\begin{enumerate}
\item $T$ is an $(m,q)$-isometry.
\item For all $x,y \in E$, the sequence $(d(T^n x, T^{n}y)^q)_{n \geq 0}$ is an arithmetic progression of order $m-1$.
\end{enumerate}
\end{proposition}

\begin{proof} We have that $T$ is an $(m,q)$-isometry if and only if Equation (\ref{isometria ms}) holds for all $x,y \in E$. Fix $x,y \in E$. For every integer $\ell \geq 0$,
$$
\sum_{k=0}^m (-1)^{m-k} {m\choose k}  d(T^k T^\ell x, T^{k} T^\ell y)^q = 0 \; ,
$$
so
\begin{equation}\label{*}
\sum_{k=0}^m (-1)^{m-k} {m\choose k}  d(T^{k+\ell} x, T^{k + \ell} y)^q = 0 \; .
\end{equation}
Therefore $(d(T^n x , T^n y)^q)_{n \geq 0 }$ is an arithmetic progression of order $m-1$.
\end{proof}

From Propositions \ref{monotony} and \ref{en em} we obtain that if $T$ is an $(m,q)$-isometry, then the sequence $(d(T^n x,T^n y)^q)_{n \geq 0}$ is eventually increasing.

The $(m,q)$-isometries are $(m+\ell,q)$-isometries, for every $\ell \geq 0$, since the arithmetic progressions of order $h$ are also arithmetic progressions of order $h+\ell$, for all $\ell \geq 0$. For this it is natural the following definition:

\begin{definicion} A map $T : E \longrightarrow E$  is called a {\rm strict $(m,q)$-isometry}  if it  is an $(m,q)$-isometry with $m=1$ or if it is an $(m,q)$-isometry, but is not an $(m-1,q)$-isometry, for $m>1$.
\end{definicion}

Now we give an analogous to Proposition \ref{en em} for strict $(m,q)$-isometries. We omit the proof since is clear.

\begin{proposition}\label{en em s} Let $T : E \longrightarrow E$ be a map. Fixed $m \geq 1$ and $q>0$, the following assertions are equivalent:
\begin{enumerate}
\item $T$ is a strict $(m,q)$-isometry.
\item For all $x,y \in E$, the sequence $(d(T^n x, T^{n}y)^q)_{n \geq 0}$ is an arithmetic progression of order $m-1$ and it is an arithmetic progression of strict order $m-1$ for some $x,y \in E$.
\end{enumerate}
\end{proposition}

By Proposition \ref{prop c d}, if $T$ is an $(m,q)$-isometry and an $(n,r)$-isometry, then
$$
(m-1)r = (n-1) q \; .
$$

Applying Proposition \ref{prop 3 options} we obtain the following result.

\begin{proposition}\label{***} Let $T : E \longrightarrow E$ be a map. Then $T$ satisfies exactly one of the following assertions:
\begin{enumerate}
\item $T$ is not an $(m,q)$-isometry for all $m \geq 1$ and $q>0$.
\item $T$ is an isometry.
\item There are unique $m \geq 2$ and $q >0$ such that $T$ is a strict $((m-1)k+1,qk)$-isometry, for every $k=1,2,3...$. Moreover, if $T$ is an strict $(n,r)$-isometry, then $n = (m-1)k+1$ and $r=qk$ for some $k=1,2,3...$
\end{enumerate}
\end{proposition}

\begin{proof} Note that if $T$ is an strict $(m,q)$-isometry, then $(d(T^k x, T^k y)^q)_{k \geq 0}$ is an arithmetic progression of order $m-1$, for all $(x,y) \in E \times E$, and $(d(T^k u, T^k v)^q)_{k \geq 0}$ is an arithmetic progression of strict order $m-1$, for some  $(u,v) \in E \times E$. Hence, for all $k \geq 1$,  $(d(T^k x, T^k y)^q)_{k \geq 0}$ is an arithmetic progression of order $k(m-1)$, for all $(x,y) \in E \times E$, and $(d(T^k u, T^k v)^q)_{k \geq 0}$ is an arithmetic progression of strict order $k(m-1)$, for some  $(u,v) \in E \times E$. Consequently $T$ is a strict $(k(m-1)+1,kq)$-isometry.
\end{proof}

The above proposition and Definition \ref{a p proper}   motivate the following definition.

\begin{definicion} An $(m,q)$-isometry $T : E \longrightarrow E$  is said to be a {\rm proper $(m,q)$-isometry}  if $m \leq n$ and $q \leq r$ whenever $T$ is a $(n,r)$-isometry.
\end{definicion}

Note that any proper $(m,q)$-isometry is a strict $(m,q)$-isometry. By \cite[Corollary 4.6]{hms}, also it is a strict $(k(m-1)+1,kq)$-isometry, for all positive integer $k$.

\;

\subsection{Distance associated to an $(m,q)$-isometry  on a metric space}

\ \par

\;

The $(m,q)$-isometries become isometries for an adequate distance. The following results are analogous to \cite{bayart}.

\begin{proposicion} \cite[Proposition 5.1]{bmmu} Let $T$ be an $(m,q)$-isometry. For $x,y \in E$ define
$$
\rho_T(x,y) := \left( \sum_{k=0}^{m-1} (-1)^{m-1-k} {m-1 \choose k}  d(T^k x, T^{k}y)^q \right)^{1/q} \; .
$$
Then $\rho_T$ is a semi-distance and moreover,
\begin{equation}\label{rho}
\rho_T (x,y)^q = (m-1)! \lim_{n \rightarrow \infty} \frac{d(T^nx , T^ny)^q}{n^{m-1}} \; .
\end{equation}
\end{proposicion}

\begin{proof} As $(d(T^nx , T^ny)_{n \geq 0})$ is an arithmetic progression of order $m-1$, by Theorem \ref{a p strict}  we can write
\begin{eqnarray}
d(T^nx,T^ny)^q &=& \ga_{m-1} n^{m-1} + \ga_{m-2} n^{m-2} + \cdots + \ga_0 \nonumber \\
&=& \sum_{k=0}^{m-1} {n \choose k}  D^k ( d(T^0 x, T^{0} y)^q) \; ,  \nonumber
\end{eqnarray}
where
$$
\ga_{m-1} = \frac{1}{(m-1)!}  D^{m-1}  ( d(T^0 x, T^{0} y)^q) = \frac{1}{(m-1)!} \left( \sum_{k=0}^{m-1} (-1)^{m-1-k} {m-1 \choose k}  d(T^k x, T^{k}y)^q \right) \; .
$$
Therefore
\begin{eqnarray}
\rho_T(x,y)^q &=& (m-1)! \; \ga_{m-1} \nonumber \\
&=& (m-1)! \lim_{n \rightarrow \infty} \frac{d(T^nx , T^ny)^q}{n^{m-1}} \; . \nonumber
\end{eqnarray}

We will show that $\rho_T$ is a semi-metric. By (\ref{rho}) it is clear that $\rho_T \geq 0$, $\rho_T(x,x)=0$ and $\rho_T (x,y) = \rho_T (y,x)$ for all $x,y\in E$. It remains to show the triangular inequality. Let $x,y,z \in E$. Then
\begin{eqnarray}
\rho_T (x,y) &=& [(m-1)!]^{1/q} \lim_{n \rightarrow \infty} \frac{d(T^nx , T^ny)}{n^{\frac{m-1}{q}}} \nonumber \\
&\leq& [(m-1)!]^{1/q} \lim_{n \rightarrow \infty} \frac{d(T^nx , T^nz)}{n^{\frac{m-1}{q}}}  + [(m-1)!]^{1/q} \lim_{n \rightarrow \infty} \frac{d(T^nz , T^ny)}{n^{\frac{m-1}{q}}} \nonumber \\
&=& \rho_T(x,z) + \rho_T(z,y) \; . \nonumber
\end{eqnarray}
So the proof is finished.
\end{proof}

By (\ref{rho}), if  $T$  is  an $(m,q)$-isometry, then $\rho_T(x,y)= \rho_T(Tx,Ty)$; that is, $T: (E,\rho_T) \longrightarrow (E,\rho_T)$ is an isometry.

\;

\subsection{Powers and products of $(m,q)$-isometries on metric spaces }

\ \par

\;

Now we apply the results of subsections 2.5 and 3.2 to $(m,q)$-isometries.

\begin{teorema} \cite[Theorems 3.15 and 3.16]{bmmu}, \cite[Theorems 3.1 and 3.6]{bdm}. Let  $T : E \longrightarrow E$ be a map.
\begin{enumerate}
\item If $T$   is a strict $(m,q)$-isometry, then any power $T^k$  is also a strict $(m,q)$-isometry.

\item If $T^c$ is a strict $(m,q)$-isometry and $T^d$ is a strict $(\ell,q)$-isometry,  then $T^e$ is an $(h,q)$-isometry, where $e$ is the greatest common divisor of $c$ and $d$, and $h$ is
the minimum of $m$ and $\ell$.

\end{enumerate}
\end{teorema}

\begin{proof} It is useful the following notation: we define the map $T \times_q T: E \times E \longrightarrow \mathbb{R}$ by $(T \times_q T)(x,y) := d(Tx,Ty)^q$.

(1) Assume that $T$   is a strict $(m,q)$-isometry. Note that this is equivalent to that the sequence $(T^n \times_q T^n)_{n \geq 0}$ is an arithmetic progression of strict order $m-1$ in $\mathbb{R}^{E \times E}$. Given an integer $k \geq 1$, $(T^{kn} \times_q T^{kn})_{n \geq 0}$ is also an arithmetic progression of strict order $m-1$ by Corollary \ref{cor}. Hence $T^k$ is a strict $(m,q)$-isometry.

(2) As $T^c$ is a strict $(m,q)$-isometry we have that $(T^{ck} \times_q T^{ck})_{k \geq 0}$ is an arithmetic progression of strict order $m-1$. Analogously, from $T^d$ a strict $(\ell,q)$-isometry we obtain that $(T^{dk} \times_q T^{dk})_{k \geq 0}$ is an arithmetic progression of strict order $\ell-1$. Applying Theorem \ref{sub a p} it results that $(T^{ek} \times_q T^{ek})_{k \geq 0}$ is an arithmetic progression of strict order $h-1$, so $T^e$ is an $(h,q)$-isometry, being $e = gcd (m,\ell)$ and $h$ the mimimun of $h$ and $k$.
\end{proof}

\begin{corollary}\label{coro} (\cite[Corollary 3.17]{bmmu}, \cite[Corollary 3.7]{bdm}).Let $E$ be a metric space and $T : E \longrightarrow E$ be a map. For positive integers $h,n,m$  and real number $q \geq 0$,
\begin{enumerate}
\item If $T$ is an $(m,q)$-isometry and $T^h$ is an isometry, then $T$ is an isometry.
\item If $T^h$ and $T^{h+1}$ are $(m,q)$-isometries, then $T$ is an $(m,q)$-isometry.
\item If $T^h$ is an $(m,q)$-isometry  and $T^{h+1}$ is an  $(n,q)$-isometry with $m<n$, then $T$ is an $(m,q)$-isometry.
\end{enumerate}
\end{corollary}

\begin{teorema}\label{mainr} (\cite[Theorem 3.14]{bmmu}, \cite[Theorem 3.3]{bdm}).
Let $E$ be a metric space,   $S,T: E \longrightarrow E$ maps such that commute, $q$ positive real number and $n,m$ positive integers. If $T$ is an $(n,q)$-isometry and $S$ is
an $(m,q)$-isometry, then $ST$ is an $(m+n-1,q)$-isometry.
\end{teorema}

\begin{proof}
Fix $x,y \in E$. Consider the sequence $(a_{i,j})_{i,j \geq 0}$ defined by $a_{i,j} :=  d(S^iT^j x , S^iT^j y)^q$. As $S$ is an $(m,q)$-isometry, the sequence $(a_{i,j})_{i \geq 0}$ is an arithmetic progression of order $m-1$, for all $j \geq 0$. Taking into account that $TS=ST$ and that $T$ is $(n,q)$-isometry, the sequence $(a_{i,j})_{j \geq 0}$ is an arithmetic progression of order $n-1$, for all $i \geq 0$. Applying Theorem \ref{product ap} we have that $(a_{i,i})_{i \geq 0}$ is an arithmetic progression of order $m+n-2$. Therefore $ST$ is an $(m+n-1,q)$-isometry. \end{proof}

\;

\subsection{$(m,q)$-isometries on Banach spaces}

\ \par

\;

Throughout this subsection, $X$ denotes a Banach space and $\| \cdot \|$ its norm, and $T: X \longrightarrow X$ a linear map.

The notion of $(m,q)$-isometry can be adapted to the setting of Banach space in the following way:

\begin{proposition} A linear map $T : X \longrightarrow X$  is an $(m,q)$-isometry  if and only if, for all $x \in X$,
\begin{equation}\label{isometria bs}
\sum_{k=0}^m (-1)^{m-k} {m\choose k}  \| T^k x \|^q = 0 \; .
\end{equation}
\end{proposition}

\begin{proof} It is enough note that $d(x,y) = \| x-y \|$ and $\| T^kx - T^ky\| =\| T^k(x-y)\|$.
\end{proof}

\begin{proposition}\label{en eb} Let $T : X \longrightarrow X$ be a linear map. Fixed $m \geq 1$ and $q>0$, the following assertions are equivalent:
\begin{enumerate}
\item $T$ is an $(m,q)$-isometry.
\item For all $x \in X$, the sequence $(\| T^n x \|^q)_{n \geq 0}$ is an arithmetic progression of order $m-1$.
\end{enumerate}
\end{proposition}

Now we give an analogous to Proposition \ref{en em} for strict $(m,q)$-isometries. We omit the proof since is clear.

\begin{proposition}\label{en eb s} Let $T : X \longrightarrow X$ be a linear map. Fixed $m \geq 1$ and $q>0$, the following assertions are equivalent:
\begin{enumerate}
\item $T$ is a strict $(m,q)$-isometry.
\item For all $x \in X$, the sequence $(\| T^n x \|^q)_{n \geq 0}$ is an arithmetic progression of order $m-1$ and it is an arithmetic progression of strict order $m-1$ for some $x \in X$.
\end{enumerate}
\end{proposition}

\;

\subsection{$m$-isometries on Hilbert spaces}

\ \par

\;

Throughout this subsection, $H$ denotes a Hilbert space and $\langle \cdot \rangle$ its inner product, $T: H \longrightarrow H$ a (linear bounded) operator and  $T^*$ its adjoint.

As Hilbert spaces are considered Banach spaces, we can apply the results given in the previous section. However, the case $q=2$ can be expressed in a special way. By this, we give the following definition:

\begin{definition} An operator $T : H \longrightarrow H$  is an {\rm $m$-isometry}  if it is an $(m,2)$-isometry.
\end{definition}

\begin{proposition}\label{pa m-isom} Let $T \in L(H)$. The following statements are equivalent:
\begin{enumerate}

\item $T$ is an  $m$-isometry; that is, for every $x \in H$, the sequence $(\| T^{k} x \|^2)_{k \geq 0}$ is an arithmetic progression of order $m-1$ in $\mathbb{R}$; that is for all $\ell \geq 0$,
$$
\sum_{k=0}^m (-1)^{m-k} {m\choose k}  \| T^{k+\ell}  x \|^2 = 0 \; .
$$
Equivalently, for every $x \in H$,
\begin{equation}\label{ap hs 2}
\sum_{k=0}^m (-1)^{m-k} {m\choose k}  \| T^{k} x \|^2 = 0 \; .
\end{equation}

\item The following operator equality
\begin{equation}\label{ap hs 3}
\sum_{k=0}^m (-1)^{m-k} {m\choose k}  T^{*k} T^k  = 0 \;
\end{equation}
holds. Equivalently, the sequence $(T^{*k}T^k)_{k \geq 0}$ is an arithmetic progression of order $m-1$ in $L(H)$; that is, for all $\ell \geq 0$,
$$
\sum_{k=0}^m (-1)^{m-k} {m\choose k}  T^{*k+\ell} T^{k+\ell}  = 0 \; .
$$

\item For every $x \in H$, the sequence $(T^{*k}T^k x)_{k \geq 0}$ is an arithmetic progression of order $m-1$ in $H$; that is, for all $\ell \geq 0$,
$$
\sum_{k=0}^m (-1)^{m-k} {m\choose k}  T^{*k+ \ell} T^{k + \ell}  x = 0 \; .
$$
Equivalently, for every $x \in H$,
$$
\sum_{k=0}^m (-1)^{m-k} {m\choose k}  T^{*k} T^k  x = 0 \; .
$$

\item For every $x,y \in H$, the sequence $(\langle T^{*k}T^k x , y \rangle)_{k \geq 0}$ is an arithmetic progression of order $m-1$ in $\mathbb{C}$; that is, for all $\ell \geq 0$,
$$
\sum_{k=0}^m (-1)^{m-k} {m\choose k}  \langle T^{*k+ \ell} T^{k+\ell}  x,y \rangle = 0 \; .
$$
Equivalently, for every $x,y \in H$,
$$
\sum_{k=0}^m (-1)^{m-k} {m\choose k}  \langle T^{*k} T^k  x,y \rangle = 0 \; .
$$

\end{enumerate}
\end{proposition}

\begin{proof} Note that, for every $x\in H$,
$$
\sum_{k=0}^m (-1)^{m-k} {m\choose k}  \| T^k x \|^2 = \sum_{k=0}^m (-1)^{m-k} {m\choose k}  \langle T^k x , T^k x \rangle = \sum_{k=0}^m (-1)^{m-k} {m\choose k}  \langle  T^{*k} T^k x , x \rangle = 0 \; .
$$
Taking into account that the operator $\sum_{k=0}^m (-1)^{m-k} {m\choose k}  T^{*k} T^k $ is self adjoint, we obtain that the conditions (1) and (2) are equivalent. The other implications are clear.
\end{proof}

\;

\subsection{A perturbation result on Hilbert spaces}

\ \par

\;

Now apply Theorem \ref{ring perturbation}, about perturbation of arithmetic progressions by nilpotents  on rings, to $m$-isometries.

We use the following notation:
$$
\beta_{m-1} (T) := \sum_{i=0}^{m-1} (-1)^{m-1-i} {m-1 \choose i} T^{*i} T^i \; .
$$

\begin{teorema}\label{hs perturbation} \cite{bmmn}. Let $T \in L(H)$ be a strict $m$-isometry. Let $Q \in L(H)$ be an $n$-nilpotent operator such that $TQ=QT$. Then $T+Q$ is an $(2n+m-2)$-isometry. Moreover, is a strict $(2n+m-2)$-isometry if and only if $Q^{*n-1} \beta_{m-1} (T) Q^{n-1} \neq 0$.
\end{teorema}

\begin{proof} We apply Theorem \ref{ring perturbation} taking $y=T^*$, $x=T$, $b=Q^*$ and $a=Q$, being both $a$ and $b$ $n$-nilpotent. Then $((T+Q)^{*k} (T+Q)^{k})_{k \geq 0}$ is an arithmetic progression of order $n+n+m-3$. Consequently $T+Q$ is a $(2n+m-2)$-isometry. Moreover, $T+Q$ is a strict $(2n+m-2)$-isometry whenever $b^{n-1} c_{m-1} a^{n-1}  \neq 0$; that is
$$
Q^{*n-1} \beta_{m-1} (T)  Q^{n-1} = Q^{* n-1} \left( \sum_{i=0}^{m-1} (-1)^{m-1-i} {m-1 \choose i} T^{*i} T^i \right) Q^{n-1} \neq 0 \; ,
$$
since $c_{m-1} = \displaystyle{\frac{1}{(m-1)!} \beta_{m-1} (T)}$ by Theorem \ref{a p strict}.
Hence the proof is finished.
\end{proof}

\;

\subsection{$n$-invertible operators on Hilbert spaces}

\ \par

\;

Suppose that $S,T : H \longrightarrow H$ are operators on the Hilbert space $H$. We say that $S$ is a {\it left $n$-inverse} of $T$ if
\begin{equation}\label{inverse}
\sum_{k=0}^{n} (-1)^{n-k} {h \choose k}  S^k T^k = 0 \; .
\end{equation}
If equation (\ref{inverse}) holds, it is said that $T$ is a {\it right $n$-inverse} of $S$; equivalently $(S^{k}T^k)_{k \geq 0}$ is an arithmetic progression of order $n-1$ in the algebra $L(H)$. We say that $S$ is an {\it left strict $n$-inverse} of $T$ if the sequence $(S^{k}T^k)_{k \geq 0}$ is an arithmetic progression of strict order $n-1$ in $L(H)$; in this case we also say that $T$ is an {\it right strict $n$-inverse} of $S$. In \cite{duggal+muller} and \cite{sid ahmed inv} are studied the right and left $n$-inverses. They are related with $m$-isometries. Note that $T$ is an $m$-isometry whenever $T^*$ is a left $m$-inverse of $T$; that is, $T$ is an right $m$-inverse of $T^*$. Many results about $m$-isometries can be extended to $m$-inverses. We give a perturbation result similar to Theorem \ref{hs perturbation}.

We denote
$$
\beta_{n-1} (S,T) := \sum_{i=0}^{n-1} (-1)^{m-1-i} {m-1 \choose i} S^{i} T^i \; .
$$
In \cite{bmmn 2} is proved the following result.

\begin{teorema}\label{hs perturbation inv} Let $S \in L(H)$ be a left strict $n$-inverse of $T$. Let $P \in L(H)$ be an $h$-nilpotent and $Q \in L(H)$ be a $k$-nilpotent such that $SP=PS$ and $TQ=QT$. Then $S+P$ is a left $(n+h+k-2)$-inverse of $T+Q$. Moreover, it is a strict left $(n+h+k-2)$-inverse if and only $P^{h-1} S^{h-k} \beta_{n-1} (S,T) Q^{k-1} \neq 0$, if $k \leq h$, or whenever $P^{h-1}  \beta_{n-1} (S,T) T^{k-h} Q^{k-1} \neq 0$, if $h \leq k$.
\end{teorema}

\begin{proof} As $S$ is a left strict $n$-inverse of $T$, we have that $(S^{k}T^k)_{k \geq 0}$ is an arithmetic progression of strict order $n-1$. In Theorem \ref{ring perturbation} we take $y=S$, $x=T$, $b=P$ and $a=Q$. Then the sequence  $((y+b)^k (x+a)^k)_{k \geq 0} = ((S+P)^k(T+Q)^k)_{k \geq 0}$ is an arithmetic progression of order $n+h+k-3$. Therefore $S+P$ is a left $(n+h+k-2)$-inverse of $T+Q$. Moreover it is strict whenever
$$
b^{h-1} y^{h-k} c_{n-1} a^{k-1} = P^{h-1} S^{h-k} \beta_{n-1}(S,T) Q^{k-1} \neq 0 \; ,
$$
if $k \leq h$, or whenever
$$
b^{h-1} c_{n-1} x^{k-h} a^{k-1} = P^{h-1}  \beta_{n-1}(S,T) T^{k-h} Q^{k-1} \neq 0 \; ,
$$
if $h \leq k$, since $c_{n-1} = \displaystyle{\frac{1}{(n-1)!} \beta_{n-1} (S,T)}$ by Theorem \ref{a p strict}.
\end{proof}

{\bf Acknowledgements:}
The first author is partially supported by grant of Ministerio  de Ciencia e Innovaci\'{o}n, Spain, proyect no. MTM2011-26538.




\end{document}